\documentclass{amsart}
\usepackage{amssymb}
\newtheorem*{theo}{Theorem}
\newtheorem{prop}{Proposition}
\newtheorem{lem}{Lemma}
\newtheorem*{cor}{Corollary}

\newcommand\bbN{\mathbb N}

\newcommand\bbR{\mathbb R}
\newcommand\bbS{\mathbb S}

\newcommand\bbZ{\mathbb Z}
\newcommand\CIc{{\mathcal{C}}^{\infty}_c}

\newcommand\CI{{\mathcal{C}}^{\infty}}

\newcommand\CmI{{\mathcal{C}}^{-\infty}}

\newcommand\Id{\operatorname{Id}}
\newcommand\Lap{\varDelta}

\newcommand\dCI{\dot{\mathcal{C}}^{\infty}}

\newcommand\ha{\frac{1}{2}}
\newcommand\loc{\operatorname{loc}}
\newcommand\nov[1]{{\frac{n}{#1}}}

\newcommand\supp{\operatorname{supp}}

\newcommand\Conr{\frac{1}{2(2\pi)^{n-1}}}
\newcommand\Conhr{2^{\ha}(2\pi)^{\frac{n-1}{2}}}

\newcommand\dH[1]{\dot{H}^{#1}}

\newcommand\fha{\frac{5}{2}}

\newcommand\hnm[1]{\frac{n-#1}{2}}
\newcommand\hnp[1]{\frac{n+#1}{2}}

\newcommand\LP{\operatorname{LP}}

\newcommand\singsupp{\operatorname{sing supp}}

\newcommand\centerpcl[3]{\vskip#2\relax\centerline{\hbox to#1{\special
  {pcl:#3}\hfil}}}

\newcommand\Hyp{\operatorname{HS}}

\newcommand\datver[1]{\def\datverp%
 {\par\boxed{\boxed{\text{Version: #1; Run: \today}}}}}

\datver{1.0B; Revised: 3-1-2008}
\begin{document}
\title[Generalized backscattering]
{Generalized backscattering and the Lax-Phillips transform}

\author{Richard Melrose AND Gunther Uhlmann}
\address{Department of Mathematics, Massachusetts Institute of Technology}
\email{rbm@math.mit.edu}
\address{Department of Mathematics, University of Washington}
\email{gunther@math.washington.edu}
\dedicatory{Dedicated to Vesselin Petkov on the occasion of his 65th birthday}
\begin{abstract} Using the free-space translation representation (modified
  Radon transform) of Lax and Phillips in odd dimensions, it is shown that
  the generalized backscattering transform (so outgoing angle $\omega
  =S\theta$ in terms of the incoming angle with $S$ orthogonal and $\Id-S$
  invertible) may be further restricted to give an entire, globally
  Fredholm, operator on appropriate Sobolev spaces of potentials with
  compact support. As a corollary we show that the modified backscattering
  map is a local isomorphism near elements of a generic set of
  potentials.
\end{abstract}
\maketitle

\tableofcontents

\section*{Introduction}

The inverse scattering problem in the two body case consists of determining
a potential $V$ by measuring the scattering amplitude $a_V(\lambda,\omega,
\theta)$ where $\lambda$ denotes the frequency of an incoming plane wave
with direction $\omega$ and $\theta$ denotes the outgoing direction. This
is an overdetermined problem except in dimension one. In this note we
consider determined problems where the set of possible angles $\theta$ and
$\omega$ is restricted to an $n-1$ dimensional subset of the complement of
the diagonal in the product of the sphere with itself. We use the time
dependent approach to scattering of Lax-Phillips \cite{LP}, \cite{P}. This
is a based on the classical wave equation rather than the time dependent or
stationary Schr\"odinger equation and therefore allows the properties of
the wave equation, especially the finite speed of propagation of the
solutions and the precise description on the propagation of singularities, to be
effectively exploited. In particular the Lax-Phillips modified Radon
transform (their free-space translation-representation), reduces the
n-dimensional problem to a one dimensional problem with lower order term
arising from the potential.

If $S$ is an n-dimensional orthogonal transformation such that $\Id-S$ is
invertible, then the modified backscattering transform determined by $S,$
for a potential $V,$ is obtained by composing the restriction of the
scattering kernel $\kappa_V(s,\omega , \theta)$ (the inverse
Fourier-Laplace transform in $\lambda$ of the scattering amplitude) to
$\omega =S\theta $ with a linear map $L_S$ (the generalized inverse of the
linearization of the map at $V=0).$ In the Main Theorem in Section 3, it is
shown that if $\dH{\hnp1}(B(\rho))$ is the Sobolev space of functions with
support in the closed ball of radius $\rho$ then
\begin{equation}
\dH{\hnp1}(B(\rho))\ni V\longmapsto
L_S(\kappa_V(s,S\theta,\theta))\in\dH{\hnp1}(B(\rho))
\label{BackScat.1}\end{equation}
is an entire and globally Fredholm non-linear map of index zero. Indeed
this map is a local isomorphism near potentials forming an open set with
complement of codimension at least two (see Proposition 2 in Section 5).

Related results, in a slightly different setting for the true
backscattering, $S=-\Id$ but including two dimensions and non-compact
supports, have been obtained by Eskin and Ralston \cite{ER1,ER2, ER3}. A
different method to prove generic uniqueness was given in \cite{S} in
dimension $3$ for compactly supported potentials. The use of hyperbolic
equations for the inverse backscattering problem has also been considered
in several papers; see for instance \cite{BLM}, \cite{GU}, \cite{Mo},
\cite{StU}.  The lecture notes \cite{MU} contain most of what we do
here. In \cite{M}, \cite{U}, we gave a sketch of the proof of the main
Theorem here for the case $S=-\Id$. The case of even dimensions $n>2$, also
for $S=-\Id$, using similar methods to \cite{MU}, \cite{M} and \cite{U},
was considered in \cite{W}. Melin has developed an alternative approach to
the inverse backscattering problem using the ultrahyperbolic equation
\cite{Me1}, \cite{Me2}.

We leave open the question of whether a map such as \eqref{BackScat.1} is a
global isomorphism, or a local isomorphism near each potential. The problem
of determining partial information of the potential, especially its
singularities, from backscattering or other (formally) determined
information has been considered in the papers \cite{GU}, \cite{J},
\cite{OPS}, \cite{SuU}, \cite{R}, \cite{RV} and in the recent preprint
\cite{BM}.

The authors thank the anonymous referee and Jeff Lagarias for comments on
the manuscript.

\section{Lax Phillips transform}

We briefly recall the approach of Lax and Phillips to scattering theory in
odd-dimensional Euclidean space. Since it suffices for the present problem
we give a simplified formulation of their theory.

The Lax Phillips theory is founded on the Radon transform: 
\begin{equation}
\begin{gathered}
Rf(s,\omega)=\int\limits_{\Hyp(s,\omega)}f(x)dH_x,\\
R:\CIc(\bbR^n)\longrightarrow\CIc(\bbR\times\bbS^{n-1})
\end{gathered}
\label{BackScat.2}\end{equation}
where $H_x$ is surface measure on $\Hyp(s,\omega)=\{x\cdot\omega =s\}.$ The
formal transpose, $R^t,$ of $R$ is given by 
\begin{equation}
\begin{gathered}
R^t:\CIc(\bbR\times\bbS^{n-1})\longrightarrow\CI(\bbR^n),\\
R^tg(x)=\int\limits_{\bbS^{n-1}}g(x\cdot\omega,\omega)d\omega.
\end{gathered}
\label{BackScat.3}\end{equation}

Of particular importance here is the fact that the Radon transform intertwines
the $n$-dimensional and the one-dimensional Laplacians (for any $n\ge2)$
\begin{equation}
R\Lap f=D^2_s Rf\quad\forall\ f\in\CIc(\bbR^n)
\label{BackScat.4}\end{equation}
where $\Delta$ is the positive Laplacian and $D_s=\frac{1}{i}\partial_s.$
Moreover there is an inversion and a Plancherel formula; for any $f,$
$g\in \CIc(\bbR^n)$ 
\begin{equation}
\begin{gathered}
f=\frac{1}{2(2\pi)^{n-1}}R^t(|D_s|^{n-1}Rf),\\
\int\limits_{\bbR^{n}}f(x)\overline{g(x)}dx=\ha\frac{1}{(2\pi)^{n-1}}
\int\limits_{\bbR\times\bbS^{n-1}}(|D_s|^{\hnm1}Rf)(s,\omega)
\overline{(|D_s|^{\hnm1}Rg)(s,\omega)}dsd\omega.
\end{gathered}
\label{BackScat.5}\end{equation}

The range of $R$ on $\CIc(\bbR^n)$ was characterized by Helgason in
\cite{Hel}. Its closure in an appropriate topology is simpler. Thus if,
$n\ge3$ is odd the operator $D_s^{\hnm1}\cdot R$ extends by continuity 
to an isometric isomorphism  
\begin{equation}
R_n=D_s^{\hnm1}\cdot R:L^2(\bbR^{n})\longrightarrow
\{k\in L^2(\bbR\times\bbS^{n-1});g(-s,-\omega)=(-1)^{\hnm1}g(s,\omega)\}
\label{BackScat.22}\end{equation}
and $R^t\cdot D_s^{\hnm1}$ extends by continuity to be its inverse.

The modified Radon transform of Lax and Phillips is defined to be
\begin{equation}
\LP\binom{u_0}{u_1}=\Conhr\left\{D_s^{\hnp1}(Ru_0)(s,\omega)
-D_s^{\hnm1}(Ru_1)(s,\omega)\right\}.
\label{BackScat.6}\end{equation}
For $n\ge 3$ odd it is an injective map
\begin{equation}
\LP:\CIc(\bbR^n)\times\CIc(\bbR^n)\longrightarrow\CIc(\bbR\times\bbS^{n-1})
\label{8.32}\end{equation}
which intertwines the free wave group and the translation group: 
\begin{equation}
\begin{gathered}
\LP\cdot U_0(t)=T_t\cdot LP,\ T_tv(s,\omega)=v(s-t,\omega),\\
U_0(t)\binom{u_0}{u_1}=\binom{u(t)}{D_tu(t)},\
(D_t^2-\Lap)u(t)=0,\ u(0)=u_0,\ D_tu(0)=u_1.
\end{gathered}
\label{BackScat.7}\end{equation}
In particular, if $u$ is the solution of the Cauchy problem for the wave
equation as in \eqref{BackScat.7} and
\begin{equation}
k(t,s,\omega)=LP\cdot
U_0(t)\binom{u_0}{u_1}\in\CI(\bbR\times\bbR\times\bbS^{n-1})
\label{8.38}\end{equation}
then $k(t,s,\omega)=k_0(s-t,\omega)$ is a solution of the first order
differential equation 
\begin{equation}
(D_t+D_s)k(t,s,\omega)=0 \text{ in }\bbR\times\bbR\times\bbS^{n-1}.
\label{BackScat.8}\end{equation}

This is in essence the free-space translation representation of Lax and
Phillips. Rather than adopting their approach of constructing a perturbed
translation representation for the wave equation with potential we use the
same `free' Lax Phillips transform and observe its effect on the solution
to the perturbed Cauchy problem
\begin{equation}
\begin{gathered}
U_V(t)\binom{u_0}{u_1}=\binom{u(t)}{D_tu(t)},\\
(D_t^2-\Lap -V(x))u(t)=0,\ u(0)=u_0,\ D_tu(0)=u_1,
\end{gathered}
\label{BackScat.9}\end{equation}
where $V\in\CIc(\bbR^n).$ Namely if 
\begin{equation}
k_V(s,t,\omega)=
LP\cdot U_V(t)\binom{u_0}{u_1}\in\CI(\bbR\times\bbR\times\bbS^{n-1}),
\label{BackScat.10}\end{equation}
then
\begin{equation}
\begin{aligned}
&(D_t+D_s)k_V(t,s,\omega)\\
&=\Conhr\left\{-D_s^{\hnm1}RD_t^2u(t,\cdot)+D_s^{\hnp3}
Ru(t,\cdot)\right\}\\
&=-\Conhr D_s^{\hnm1}R[V(\cdot)u(t,\cdot)].
\end{aligned} 
\label{BackScat.11}\end{equation}
Using the inversion formula it follows that 
\begin{equation}
(D_t+D_s)k_V(t,s,\omega)+V_{\LP}k_V(t,s,\omega)=0
\label{BackScat.12}\end{equation}
where $V_{\LP}$ is an operator on $\CI(\bbR\times\bbS^{n-1}):$ 
\begin{equation}
V_{\LP}=\Conr D_s^{\hnm1}\cdot R\cdot V\cdot R^t D^{\hnm3}_s.
\label{BackScat.13}\end{equation}
Thus if $\supp(V)\subset\{|x|\le\rho\}$ the operator $V_{\LP}$
defined by \eqref{BackScat.13} has Schwartz kernel $V_{\LP}(s,\omega,s',\omega')$
supported in the region 
\begin{equation}
\supp(V_{\LP})\subset\left\{(s,\omega',s',\omega )\in
\bbR\times\bbS^{n-1}\times\bbR\times\bbS^{n-1};|s|,|s'|\le\rho\right\}.
\label{BackScat.14}\end{equation}

There is a unique fundamental solution, which is to say a distribution satisfying 
\begin{equation}
\begin{gathered}(D_t+D_s+V_{\LP})E_{\LP}(t,s,\omega;s',\theta)=0\\
E_{\LP}(0,s,\omega;s',\theta)=\delta(s-s')\delta_\theta(\omega).
\end{gathered}
\label{BackScat.15}\end{equation}
Standard properties of the wave equation imply that
\begin{equation}
\begin{gathered}
\singsupp(E_{\LP})\subset\\
\{s'-s+t=0,\theta=\omega\}\cup
\{s'+s+t=0,\theta=-\omega,|s|\le\rho,|s'|\le\rho\}.
\end{gathered}
\label{BackScat.16}\end{equation}

>From this it follows that the continuation problem can also be solved, so
for each $\theta\in\bbS^{n-1}$ there is a unique distribution
\begin{equation}
\alpha(t,s,\omega,\theta)\in
\CmI(\bbR\times\bbR\times\bbS^{n-1}\times\bbS^{n-1}),
\label{BackScat.17}\end{equation}
satisfying 
\begin{equation}
\begin{gathered}
(D_t+D_s)\alpha+V_{\LP}\alpha=0 \text{ in
  }\bbR\times\bbR\times\bbS^{n-1}\text{ and}\\
\alpha(t,s,\omega;\theta)=\delta(s-t)\delta_\theta(\omega)\text{ in }t<-\rho
\end{gathered}
\label{BackScat.19}\end{equation}
where $\rho=\sup\{|x|;x\in\supp(V)\}.$

It follows that 
\begin{equation}
\alpha(t,s,\omega;\theta)=\kappa_V(t-s,\omega;\theta)\text{ in }s>\rho
\label{BackScat.20}\end{equation}
where $\kappa_V\in\CmI(\bbR\times\bbS^{n-1}\times\bbS^{n-1})$ is the\/ {\it
 scattering kernel}. Here one can think of $\alpha$ as the free wave 
\begin{equation}
\alpha_0(t,s,\omega;\theta)=\delta(s-t)\delta_\theta(\omega)
\label{BackScat.21}\end{equation}
propagating in from the left and striking the `potential' which is confined to
the region $|s|\le\rho.$ Once it has passed through the potential it again
freely propagates to the right. Thus the kernel $\kappa_V(t,\omega;\theta)$
represents the end result of the interaction. 

The scattering amplitude in the ordinary sense is the Fourier-Laplace
transform of $\kappa_V$ continued to the real axis. We define the
generalized backscattering transform below directly from $\kappa_V.$

\section{Sobolev bounds}

We will consider potentials $V$ with fixed support and finite Sobolev
regularity. So, for $\rho\in(0,\infty),$ set
\begin{multline}
\dH{\hnp1}(B(\rho))=\{V\in L^2(\bbR^n); V(x)=0 \text{ in }|x|>\rho,\\
D^\alpha V\in L^2\ \forall\ |\alpha|\le\hnp1\}.
\label{12.3}\end{multline}

The choice of Sobolev order here is not critical; it is convenient that
$\hnp1$ is an integer and rather more important that $\hnp1>\nov2.$ The
latter condition means that $\dH{\hnp1}B(\rho))$ is an {\it algebra}. In
fact the usual Sobolev spaces are then modules over these for an
appropriate range of orders.

\begin{lem}\label{12.4} (Gagliardo-Nirenberg, see \cite{Ad})\ For
any $k\in\bbN$ with $k>n/2$ and any $s\in\bbR$ satisfying $-k\le s\le k$
\begin{equation}H^k(\bbR^n)\cdot H^s(\bbR^n)\subset H^s(\bbR^n).
\label{12.5}\end{equation}
In particular, if $s\in\bbR$ and $-\hnp1\le s\le\hnp1,$ then
\begin{equation}
\dH{\hnp1}(B(\rho))\cdot
H^s(\bbR^n)\subset\dH s(B(\rho)).
\label{12.6}\end{equation}
\end{lem}
\noindent

\begin{lem}\label{12.7} For any $k$ ($\in\bbZ$ for simplicity) the
normalized Radon transform in \eqref{BackScat.22} gives a bounded map
\begin{equation}\begin{aligned}
R_n:\dH k(B(\rho))\longrightarrow\dH k([-\rho,\rho ]\times\bbS^{n-1})
=\{u\in & H^k(\bbR\times\bbS^{n-1});\\
& u(s,\theta)=0\text{ in }|s|>\rho \}.\end{aligned}
\label{12.8}\end{equation}
\end{lem}

\begin{proof} For $k=0,$ this is \eqref{BackScat.22} which is a
consequence of the $L^2$ boundedness of the Fourier transform. Consider the
case $k>0.$ We know that $R$ (and hence $R_n$) intertwines $\Lap$ with
$D^2_s.$ Thus if $f\in\CIc(\bbR^n)$ then
\begin{equation}
D^2_sR_n f=R_n\Lap f.
\label{12.9}\end{equation}
Since $R_n$ is a partial isometry on $L^2,$
\begin{equation}
\langle R_nf, D^2_s R_n f\rangle_{L^{2}}=
\langle\Lap f,f\rangle.
\label{12.10}\end{equation}
By continuity then, $f\in\dH1(B(\rho))\Longrightarrow D_sR_nf\in L^2.$
Repeating this argument a finite number of times shows that
\begin{equation}
f\in\dH k(B(\rho))\Longrightarrow D_s^j R_n f\in
L^2([-\rho,\rho]\times\bbS^{n-1})\quad0\le j\le k.
\label{12.11}\end{equation}
To get tangential regularity, suppose that $W$ is a $\CI$ vector field on the
sphere. Then
\begin{equation}
\begin{aligned}
&WR_n f(s,\theta)=c_n D_s^{\hnm1} W\int\delta(s-x\cdot\theta)f(x)dx\\
&=\sum\limits_{j=1}^n q_j(\theta)D_sR_n(x_jf),\quad
W(x\cdot\theta)=\sum\limits_{j=1}^nx_jq_j(\theta).
\end{aligned}
\label{12.12}\end{equation}
Thus $WR_nf\in L^2.$ Repeating this argument we conclude that \eqref{12.8}
holds for $k\ge 0.$

The same type of argument applies to $R^t_n.$ Thus
\begin{equation}
R_n^t u(x)=c_n\int\limits_{\bbS^{n-1}}\delta(s-x\cdot\omega)
D_s^{\hnm1} u(s,\omega)ds
\label{12.13}\end{equation}
is bounded from $L^2([-\rho,\rho ]\times\bbS^{n-1})$ into $L^2(B(\rho)).$
Direct differentiation therefore shows that it is bounded from
$H^k([-\rho,\rho ]\times\bbS^{n-1})$ into $H^k(B(\rho))$ for $k\in\bbN.$ By
duality it follows that \eqref{12.8} holds for $k\in-\bbN,$ and
hence for all $k\in\bbZ$ as claimed.
\end{proof}

Note that, from the proof above,
\begin{gather}
{\begin{aligned}
R^t:\{u\in\CmI(\bbR\times\bbS^{n-1});&
D^j_s u\in L^2_{\loc}(\bbR\times\bbS^{n-1}),\ 0\le j\le k\}\\
&\longrightarrow H^k(B(\rho)) \text{ if } k\ge 0, \text{ and }
\end{aligned}}
\label{12.15}\\
{\begin{aligned}
R^t:\{u\in\CmI(\bbR\times\bbS^{n-1});&u\in L^2_{\loc}(\bbR\times\bbS^{n-1})+
D^{-k}_s L^2(\bbR\times\bbS^{n-1})\}\\
&\longrightarrow H^k(B(\rho))\text{ if } k\le 0.
\end{aligned}}
\label{12.16}
\end{gather}
That is, one does not need tangential regularity to ensure the regularity of
$R_n^tf$ in a compact set.

\begin{lem}\label{12.17} For any $k\in\bbZ$ satisfying
$\hnm1\ge k\ge-\hnp3,$ and any potential $V\in\dH{\hnp1}(B(\rho)),$
$V_{\LP}$ gives a bounded map
\begin{equation}V_{\LP}: H^k(\bbR\times\bbS^{n-1})\longrightarrow
H^{k+1}(\bbR\times\bbS^{n-1}).
\label{12.18}\end{equation}
\end{lem}

\begin{proof} Recall that $V_{\LP}=c^2_n D_s^{\hnm1}R\cdot V\cdot R^t
D_s^{\hnm3}.$ From \eqref{12.15},
\begin{equation}
R^t D_s^{\hnm3}: H^k(\bbR\times\bbS^{n-1})\longrightarrow H^{k+1}(B(\rho)).
\label{12.19}\end{equation}
Then, from Lemma~\ref{12.4}, multiplication by $V$ maps into the space
$\dH{k+1}(B(\rho))$ and from Lemma~\ref{12.7},
$D_s^{\hnm1}R$ maps into $\dH{k+1}([-\rho,\rho]\times\bbS^{n-1}).$
\end{proof}

\section{Generalized backscattering transform}

We shall apply these regularity estimates to show that a `modified
backscattering transform,' in which `excess' information has been
discarded, extends by continuity to $\dH{\hnp1}(B(\rho)).$

Let $\pi_{S,\rho}$ be the orthogonal projection, in
$H^2([-2\rho,2\rho]\times\bbS^{n-1}),$ onto the closure of the range of
$D_s^{\hnm3}R_n$ applied to
$(\Id-S)^*\dH{\hnp1}(B(\rho))=\dH{\hnp1}((\Id-S)B(\rho))$ using
Lemma~\ref{12.7}; let $P_{S,\rho}$ be the range of $\pi_{S,\rho}.$ For
$V\in\CIc(\bbR^n)$ we know that the scattering kernel $\kappa_V,$ has
support in $\{s\ge-2\rho \}.$ We will `cut off the tail' where $s>2\rho$
and project the rest using $\pi_{S,\rho}.$ Thus, consider the combined
restriction, differentiation and projection map
\begin{multline}
\chi_\rho:\CI(\bbR\times\bbS^{n-1})\overset{D_s^{\hnm3}}\longrightarrow
\CI([-2\rho,2\rho ]\times\bbS^{n-1})\\
\overset{\pi_{S,\rho}}\longrightarrow
\overline{D_s^{\hnm3}R_n(\dH{\hnp1}((\Id-S)B(\rho))}
\subset\dH2([-2\rho,2\rho ]\times\bbS^{n-1}).
\label{12.20}\end{multline}

Now, for $V\in\CIc(\bbR^n)$ we know that
\begin{equation}
\singsupp \kappa_V\subseteq\{s=0,\theta=\omega \}.
\label{12.21}\end{equation}
Thus the generalized backscattering kernel,
$\kappa_V(s,S\theta,\theta)\in\CI(\bbR\times\bbS^{n-1}).$ We can therefore
apply \eqref{12.20} to define the {\it modified (and generalized)
  backscattering transform}
\begin{equation}
\beta_S:\dCI(B(\rho))\ni V\longmapsto
\chi_\rho[\kappa_V(s,S\theta,\theta)]\in P_{S,\rho}\subset
\dH2([-2\rho,2\rho]\times\bbS^{n-1}).
\label{12.22}\end{equation}

\begin{theo}[Main Result]\label{12.23} For any orthogonal transformation
  $S,$ such that $\Id-S$ is invertible, the modified backscattering
  transform \eqref{12.22} extends, by continuity, to
\begin{equation}
\beta_S:\dH{\hnp1}(B(\rho))\longrightarrow P_{S,\rho}\subset
\dH2([-2\rho,2\rho]\times\bbS^{n-1})
\label{12.24}\end{equation}
which is entire analytic. More precisely, it can be written
\begin{equation}
\beta_S(V)=\sum\limits_{j=1}^\infty\beta^j_S(V_,\dots V)
\label{12.25}\end{equation}
where
\begin{equation}
\beta_S^1:\dH{\hnp1}(B(\rho))\longrightarrow P_{S,\rho}\subset
\dH2([-2\rho,2\rho ]\times\bbS^{n-1})
\label{12.26}\end{equation}
is a linear isomorphism and for each $j\ge 2$
\begin{equation}
\beta_S^j:[\dH{\hnp1}(B(\rho))]^j\longrightarrow P_{S,\rho}\cap\dH{\fha}
([-2\rho ,2\rho ]\times\bbS^{n-1})
\label{12.27}\end{equation}
is symmetric and satisfies, for each $0\le\epsilon\le\ha,$
\begin{equation}\|\beta_S^j(V,\dots,V)\|_{\fha-\epsilon}\le
\frac{C^{j+1}\|V\|^{j}}{(j!)^{2\epsilon}}.
\label{12.28}\end{equation}
\end{theo}

As we shall describe below, this implies that $\beta_S$ is almost
everywhere a local isomorphism. It is not known, at least to the authors,
whether $\beta_S$ is a global isomorphism (for any admissible $S,$ in
particular $S=-\Id).$ Nor indeed is it known whether the differential of
$\beta_S,$ at $V\in\dH{\hnp1}(B(\rho))$ is always invertible --
although it {\it is} Fredholm. Nor is there a conjectural characterization
of the singular points.

The Taylor expansion \eqref{12.25} for the modified backscattering transform
is closely related to the Born approximation. This in turn is just the
Neumann (or perhaps better to say Volterra) series for the solution of the
(Radon-transformed) wave equation.

Formally at least, the solution to \eqref{BackScat.19} can be expanded as a series
\begin{equation}
\begin{gathered}
\alpha=\delta(s-t)\delta_\theta(\omega)+\sum\limits_{j=1}^\infty(-1)^j\alpha_j,\
\alpha_j=[(D_t+D_s)^{-1}V_{\LP}]^j\alpha_0,\ j\ge1\\
\alpha_0=\delta(s-t)\delta_\theta(\omega).
\end{gathered}
\label{12.30}\end{equation}
Here $(D_t+D_s)^{-1}$ is the inverse of the free forcing problem
\begin{equation}
(D_t+D_s)u=f,\ f=0\text{ in }s<-\rho,\ u=0\text{ in }s<-\rho
\Longrightarrow u=(D_t+D_s)^{-1}f.
\label{12.31}\end{equation}

We proceed to show that, for any $V\in\dH{\hnp1}(B(\rho)),$ the series
\eqref{12.30} converges.

\begin{prop}\label{12.32} For any
$V\in\dH{\hnp1}(B(\rho)),$ $T<\infty$ and $k\in\bbZ,$ with
$-\hnp3\le k\le\hnp1,$ $(D_t+D_s)^{-1} V_{\LP}$ is bounded as an operator on
\begin{equation}
\dH k_{T,\rho}=\left\{ f\in\dH k([-\infty,T]_t\times [-\rho,\rho
]_s\times\bbS^{n-1});f=0 \text{ in } t<-\rho\right\}
\label{12.33}\end{equation}
and for some $C=C(T)$
\begin{equation}
\|[(D_t+D_s)^{-1}
V_{\LP}]^j\|_{H^{k}}\le\frac{C^{j+1}\|V\|^{j}}{j!},
\label{12.34}\end{equation}
where $\|V\|$ is the norm in $\dH{\hnp1}(B(\rho)).$
\end{prop}

\begin{proof} Since $t$ is a parameter in the action of $V_{\LP}$ and
$(D_t+D_s)^{-1}$ is bounded on any Sobolev space the boundedness is clear
from Lemma~\ref{12.17}. Only the Volterra-type estimate
\eqref{12.34} needs to be shown. To carry out this estimation it is
convenient to introduce $D_t+D_s$ and $D_s$ as coordinate vector fields,
i.e.\ change coordinates to
\begin{equation}
t'=t,\ s'=s-t.
\label{12.35}\end{equation}
The operators are transformed as follows
\begin{equation}
D_t+ D_s\longmapsto D_{t'},\ V_{\LP} \longmapsto V'_{\LP}(t', s',
D_{s'})
\label{12.36}\end{equation}
where $V'_{\LP}$ is still a non-local operator in $s',$ but now depending on
$t'$ as a parameter, i.e.
\begin{equation}
V'_{\LP}u(t',s') \text{ depends only on } u(t',\cdot).
\label{12.37}\end{equation}
The iterated operator is therefore
\begin{equation}
\left( D_{t'}^{-1} V'_{\LP}\right)^j.
\label{12.38}\end{equation}
Applying this $|k|+1$ times to $H^k$ gives a bounded map into the space
$$
C^0([-\rho,T];H^k(\bbS^{n-1}\times\bbR_{s'})).
$$
Then, integration in $t'$ and
continuity of $V'_{\LP}$ shows that
\begin{equation}
\|(D^{-1}_{t'} V'_{\LP})^{j+|k|+1}
u\|_{H^{k}(\bbS^{n-1}\times\bbR_{s'})}(t')\le\frac{C(t'+\rho)^{j}}{j!}.
\label{12.39}\end{equation}
This gives \eqref{12.34}.
\end{proof}

Of course from Lemma~\ref{12.17} we know that, if $-\hnp3\le k\le\hnm1,$
\begin{equation}
(D_t+D_s)^{-1}V_{\LP}:\dH k_{T,\rho}\longrightarrow\dH{k+1}_{T,\rho}
\label{12.40}\end{equation}
Since
\begin{equation}
\delta(t-s)\delta_{\theta}(\omega)\in
H_{\loc}^{-\hnp1}(\bbR^2\times\bbS^{n-1}\times\bbS^{n-1})
\end{equation}
it follows that
\begin{equation}
\alpha_j\in H_{\loc}^{-\hnp1+\min(j,n+1)}(\bbR^2\times\bbS^{n-1}
\times\bbS^{n-1}).
\label{12.41}\end{equation}

Consider the successive terms, $\alpha_j,$ in \eqref{12.30}. Since
$V_{\LP}$ always restricts supports to $[-\rho,\rho]$ in $s,$
\begin{equation}
\supp(\alpha_j)\subseteq\{t\ge-\rho\}\cap\{s\ge-\rho\}\cap\{t-s\ge-2\rho\}
\cap\{t-s\le 2j\rho\}.
\label{12.42}\end{equation}


To get the expansion \eqref{12.25} we need to use \eqref{12.30} and then
project each term with $\chi_\rho,$ after restricting to $s=\rho,$
$\omega=S\theta$ (and shifting in $t)$ to get the scattering kernel. Thus if
\begin{equation}
\kappa_j(s,\omega,\theta)=\alpha_j(s-\rho,\rho,\theta,\omega)
\label{12.43}\end{equation}
then
\begin{equation}
\beta_S^j(V)=\chi_\rho[\kappa_j(s,S\theta,\theta)].
\label{12.44}\end{equation}

Since, as a function of $t-s,s,\omega$ and $\theta,$ $\alpha_j$ is
independent of $s$ in $s>-\rho$ it follows from \eqref{12.41} that
\begin{equation}
\kappa_j\in
H^{-\hnp1+\min(j,n+1)}([-2\rho,T)\times\bbS^{n-1}\times\bbS^{n-1})\text{
for any T.}
\label{12.45}\end{equation}
Restricting to $\omega=S\theta,$ a submanifold of codimension $n-1$ shows
that
\begin{equation}
\kappa_j(s,S\theta,\theta)\in H^1([-2\rho,T)\times\bbS^{n-1})\text{ if }j\ge
n+1.
\label{12.46}\end{equation}
Moreover, to get \eqref{12.46} we only use the regularity property
\eqref{12.40} for the first $n+1$ factors in \eqref{12.38}. Thus we
conclude that the map
\begin{equation}
\dH{\hnp1}(B(\rho))\longrightarrow \sum\limits_{j\ge n+1}
\kappa_j(s,S\theta,\theta)\in H^1([-2\rho,T)\times\bbS^{n-1})
\text{ is entire}
\label{12.47}\end{equation}
for each $\rho.$ This is a good deal weaker than we need to prove
the Theorem. Obviously we need to examine the first $n+1$
terms in the Taylor series of $\beta$ at $V=0$ to show that this polynomial
in $V$ is defined and in any case we have to show that the whole map $\beta_S$
takes values in $H^2$ rather than $H^1.$ Nevertheless we shall use
\eqref{12.47} because it allows us to prove that $\beta$ is entire, with
values in the good space (essentially because of Pettit's theorem).

\section{Proof of the main result}

First we examine 
\begin{equation}
\kappa_S^1(s,S\theta,\theta)=\alpha_1(s-\rho,\rho,S\theta,\theta)
\label{12.48}.\end{equation}
This already has support in $[-2\rho, 2\rho ].$ We wish to show that this,
the linear, term is as claimed in \eqref{12.26}. We proceed to compute
$\kappa_1$ explicitly. It is convenient to take the Fourier transform in $s:$
\begin{equation}
\widehat{k_1}(\lambda,\omega,\theta)=\int\limits_{-\infty }^\infty
e^{-i\lambda t}\kappa_1(t,\omega,\theta)dt=
\widehat{\alpha_1}(\lambda,\rho,\omega,\theta)e^{i\lambda\rho }.\end{equation}
>From the definition of $\alpha_1,$ this gives
\begin{equation}
\begin{aligned}
\widehat{\kappa_1}(\lambda,\omega,\theta)=e^{i\lambda\rho }
\int\limits_{-\infty }^\infty \int e^{-i\lambda(\rho-s')}\left[ V_{\LP}
e^{-i\lambda s }\delta_ \theta(\omega)\right] ds'.\\
=c^2_n\int e^{i\lambda s}D_s^{\hnm1}\int\limits_{x\cdot\omega=s} V(x)
\lambda^{\hnm3} e^{-i\lambda x\cdot\theta }dxds.
\end{aligned}
\end{equation}
Integrating by parts we get
\begin{equation}
\widehat{\kappa_1}(\lambda,\omega,\theta)=c_n^2\lambda^{n-2}\int
e^{i\lambda x\cdot(\omega-\theta)} V(x) dx.
\label{12.49}\end{equation}
Setting $\omega=S\theta$ we find
\begin{equation}
\widehat{\kappa_S^1}(\lambda,S\theta,\theta)=
c_n^2\lambda^{n-2}\hat{V}(\lambda(\Id-S)\theta).
\label{12.50}\end{equation}
Thus $\widehat{\beta_S^1(V)}$ is the ($n$-dimensional) Fourier transform of
$2^{-n}V((\Id-S)^{-1}x)=\tilde V_S.$ Hence,
\begin{equation}
\beta_S^1=c_n D_s^{\hnm3} R_n\tilde V_S
\label{12.51}\end{equation}
shows that $\beta_S^1$ maps into $\dH2([-2\rho,2\rho ]\times\bbS^{n-1}).$ It
is obviously an isomorphism onto $D_s^{\hnm3}R_n\dH{\hnp1}((\Id-S)B(\rho))$
(which is closed) as claimed.

Next we proceed to find a formula generalizing \eqref{12.49} to the higher
derivatives at zero. From \eqref{12.42} we see that, for $s$ bounded above,
the support of each $\alpha_j$ is compact in $t.$ After taking the Fourier
transform in $t,$ the iterative definition \eqref{12.30} becomes
\begin{equation}
\widehat{\alpha_j}(\lambda,s,\omega,\theta)=
(D_s+\lambda)^{-1}R_n[V\cdot Q_
\lambda]^{j-1}VR^tD_s^{(n-3)/2}e^{-is\lambda}\delta_\theta(\omega),
\label{12.52}\end{equation}
where
\begin{equation}
Q_\lambda=R_n^tD_s^{-1}(D_s+\lambda)^{-1}R_n.
\label{12.53}\end{equation}
Here $D_s^{-1},$ and $(D_s+\lambda)^{-1}$ mean integration from $s=-\infty,$
i.e.\ the inverse preserving vanishing to the left.

\begin{lem}\label{12.54} Acting from $\CIc(\bbR^n)$ to $\CI(\bbR^n),$
$Q_\lambda=(\Lap-\lambda^2)^{-1}$ is the analytic extension of the `free
resolvent' defined as a bounded operator on $L^2$ for $\Im\lambda<0.$
\end{lem}

\begin{proof} This formula can be deduced from the modified Radon
transform of Lax and Phillips. We know that this intertwines the wave group
$U_0(t)$ with the translation group, so conjugates the infinitesimal
generator of one to that of the other
\begin{equation}
c_n(D_s^{\hnm1}R,D_s^{\hnp1}R)\begin{pmatrix} 0 &-1\\ \Lap
&0\end{pmatrix}
=D_s(D_s^{\hnm1}R,D_s^{\hnp1}R).
\label{12.55}\end{equation}
For $\Im\lambda<0,$ so in the resolvent set, it follows that
\begin{equation}
c_n^2R^tD_s^{\hnm3}(D_s+\lambda)^{-1}D_s^{\hnm1}=(\Lap-
\lambda^2)^{-1}.
\label{12.56}\end{equation}
This proves the lemma.
\end{proof}

Inserting the integral expression for $(D_s+\lambda)^{-1}$ into
\eqref{12.52} gives
\begin{multline}
\widehat{\alpha_j}(\lambda,s,\omega,\theta)=
c_n^2\int\limits_{-\infty}^s
e^{-i\lambda(s-s')}D_{s'}^{\hnm1}\int\limits_{x\cdot\omega=s'}V\cdot
Q_\lambda\cdot V\cdots\\
Q_\lambda\cdot[V(\bullet)(-\lambda)^{\hnm3}
e^{-i\lambda\bullet\cdot\theta}]dH_xds'.
\label{12.57}\end{multline}
>From \eqref{12.43}, by setting $s=\rho$ and integrating by parts we find
\begin{equation}
\widehat{\kappa_j}(\lambda,\omega,\theta)=
c_n^2(-1)^{\hnm3}\lambda^{n-2}\int\limits_{\bbR^n}e^{i\lambda\omega\cdot x}
V(x)[Q_\lambda\cdots Q_\lambda\cdot V(\bullet)
e^{-i\lambda\theta\cdot\bullet}](x)dx.
\label{12.58}\end{equation}
Restricting to backscattering, $\omega=S\theta,$ this gives
$\widehat{\kappa_S^j}$ in a form similar to \eqref{12.50}. Since $\kappa_j$
has support in $[-2\rho,2j\rho]$ its regularity can be deduced from its
Fourier-Laplace transform with $\Im\lambda=-1.$ Thus we need to examine the
growth in $\lambda$ of
\begin{multline}
\widehat{\kappa_j}(\lambda,S\theta,\theta)=\\
c_n^2\lambda^{n-2}
\int\limits_{\bbR^{jn}}e^{i\lambda \theta\cdot(S^tx^{(1)}-x^{(j)})}
V(x^{(1)})Q_\lambda(x^{(1)}-x^{(2)})V(x^{(2)})\dots\\
\dots Q_\lambda(x^{(j-1)}-x^{(j)})V(x^{(j)})(x)dx^{(1)}\dots dx^{(j)}
\label{12.59}\end{multline}
where there are $j-1$ factors of the free resolvent, $Q_\lambda,$ and $j$
factors of $V.$ As a convolution operator $Q_\lambda$ has kernel
\begin{equation}
Q_\lambda(y)=(2\pi)^{-n}\int
e^{iy\cdot\eta}(|\eta|^2-\lambda^2)^{-1}d\eta.
\label{12.60}\end{equation}
Inserting this into \eqref{12.59} gives
\begin{multline}
\widehat{\kappa_j}(\lambda,S\theta,\theta)=\\
c_n^2\int V(x^{(1)})V(x^{(2)})\dots V(x^{(j)})
\prod\limits_{\ell=1}^{j-1}(|\eta^{(\ell)}|^2-\lambda^2)^{-1}\\
\times\exp[i(S^tx^{(1)}-x^{(j)})\cdot\xi
+i(x^{(1)}-x^{(2)})\cdot\eta^{(1)}+\dots+i(x^{(j-1)}-x^{(j)})\cdot
\eta^{(j-1)}]\\
dx^{(1)}\dots dx^{(j-1)}d\eta^{(1)}\dots d\eta^{(j-1)}
\label{12.61}\end{multline}
where $\xi=\lambda\theta.$

Carrying out the $x$-integrals in \eqref{12.61} gives
\begin{equation}
\begin{gathered}
\widehat{\kappa_j}(\lambda,S\theta,\theta)\\
{\begin{aligned}= c_n^2\lambda^{n-2}&\int
\hat{V}(-S^t\xi-\eta^{(1)})\hat{V}(\eta^{(1)}-\eta^{(2)})\dots
\hat{V}(\eta^{(j-2)}-\eta^{(j-1)})\hat{V}(\eta^{(j-1)}+\xi)\\
&\prod\limits_{\ell=1}^{j-1}(|\eta^{(\ell)}|^2-\lambda^2)^{-1}
d\eta^{(1)}\dots d\eta^{(j-1)}.\end{aligned}}\end{gathered}
\label{12.62}\end{equation}
Apart from the factors arising from the resolvent this is an iterated
convolution. Since $\Im\lambda=-1,$ the resolvent factors are non-singular.
Using the obvious estimates
\begin{equation}
|(|\eta|^2-\lambda^2)^{-1}|\le c(1+|\eta|+|\lambda|)^{-1}.
\label{12.63}\end{equation}
and
\begin{equation}
(1+|\eta'|+|\lambda|)^{-1}(1+|\eta|+|\lambda|)^{-1}\le(1+|\eta-\eta'|)^{-1}
\label{12.64}\end{equation}
the right side of \eqref{12.62} can be estimated to give
\begin{equation}
\begin{gathered}
|\widehat{\kappa_j}(\lambda,S\theta,\theta)|\le
C^{j+1}|\lambda|^{n-2}\times\\
{\begin{aligned}\int\hat{\Phi}&(-S^t\xi-\eta^{(1)})
\hat{\Phi}(\eta^{(1)}-\eta^{(2)})\\
&\dots\hat{\Phi}(\eta^{(j-2)}-\eta^{(j-1)})\hat{\Phi}(\eta^{(j-1)}+\xi)
d\eta^{(1)}\dots d\eta^{(j-1)},\end{aligned}}
\end{gathered}
\label{12.65}\end{equation}
where
\begin{equation}\hat{\Phi}(\eta)=|\hat{V}(\eta)|(1+|\eta|)^{-\ha}.
\label{12.66}\end{equation}
Thus
\begin{equation}
\|\Phi\|_{H^{(n+2)/2}}\le\|V\|_{H^{(n+1)/2}}.
\label{12.67}\end{equation}
First translating the variables of integration to $\eta^{(\ell)}+\xi$ we find
that the right side of \eqref{12.65} is the Fourier transform of a product of
functions, so using Lemma~\ref{12.4} repeatedly (and taking into
account the factor of $\lambda^{n-2}$ and the invertibility of $S^t-\Id)$
\begin{equation}
\|\kappa_j(s,S\theta,\theta)\|_{H^{\fha}([-2\rho,2\rho ]\times\bbS^{n-1}}\le
C^{1+j}\|V\|_{H^{(n+1)/2}}.
\label{12.68}\end{equation}

This gives the desired continuity \eqref{12.27} and estimates \eqref{12.28}
for $\epsilon=0.$ Moreover the estimates \eqref{12.49} give \eqref{12.28}
for $\epsilon=\ha$ and large (hence all) $j.$ The estimates for all
$\epsilon\in [0,\ha]$ then follow by interpolation between Sobolev spaces,
i.e.
\begin{equation}\|u\|_{\fha-\epsilon}\le C\|u\|_{2}^{2\epsilon}
\|u\|_{\fha}^{1-2\epsilon}\quad\forall\ \epsilon\in[0,\ha].
\label{12.69}\end{equation}
This completes the proof of Theorem~\ref{12.23}.

It may be that the estimates centered on \eqref{12.63} can be improved to
give the exponential type estimates \eqref{12.28} directly and with values
in $H^{\fha}.$ If the original regularity $(n+1)/2$ for $V$ is increased by
$p$ then the regularity of the derivatives $\beta_S^j$ in \eqref{12.28} can
also be increased by $p.$

Note that the map $\beta_S$ in \eqref{12.44} is defined by projection onto
the range of the linearization of $V\longmapsto \kappa_V(s,S\theta,\theta)$
at $V=0.$ The linearization has been shown to be an injective Fredholm map,
i.e.\ is an isomorphism onto its (closed) range, so its generalized inverse
is a bounded map
\begin{equation}
L_S:\dH2([-2\rho,2\rho]\times\bbS^{n-1})\longrightarrow
\dH2([-2\rho,2\rho]\times\bbS^{n-1})\longrightarrow \dH{\hnp1}(B(\rho)).
\label{BackScat.23}\end{equation}
The map in \eqref{BackScat.1} is then 
\begin{equation}
L_S\beta_S(V)=L_S(\kappa_V(s,S\theta ,\theta))\text{ on }\dH{\hnp1}(B(\rho))
\label{BackScat.24}\end{equation}
which is therefore an entire map with linearization the identity at $0$ and
derivative at all other points a compact perturbation of the identity.

\section{Fredholm property}

\begin{prop}\label{12.72} There is a closed subset
$G(\rho)\subset\dH{\hnp1}(B(\rho))$ which is of codimension at least two
(i.e.\ locally orthogonal projection from $G(\rho)$ onto some subspace of
codimension two is at most $p$-to-$1$ for some fixed $p\in\bbN)$ such that for
each $V'\in[\dH{\hnp1}(B(\rho))\setminus G(\rho)]$ there exists $\epsilon>0$
such that the map
\begin{equation}
\beta_S:\left\{V\in\dH{\hnp1}(B(\rho));\|V-V'\|<\epsilon\right\}
\longrightarrow\dH2([-2\rho,2\rho]\times\bbS^{n-1})
\label{12.73}\end{equation}
is an isomorphism onto its image.
\end{prop}

\begin{proof} The set $G(\rho)$ is defined to consist of those
$V\in\dH{\hnp1}(B(\rho))$ such that the derivative of $\beta_S$ with
respect to $V$ is not an isomorphism. Certainly \eqref{12.73} holds for
points in the complement of $G(\rho)$ by the implicit function theorem,
applied in the Sobolev space. Thus we need to show that $G(\rho)$ so
defined has codimension at least $2,$ since the density of the complement
certainly follows from this. The derivative of $\beta_S$ with respect to
$V$ is a linear map
\begin{equation}
\beta_1+\gamma(V):\dH{\hnp1}(B(\rho))\longrightarrow\dH2([-2\rho,2\rho]
\times\bbS^{n-1})
\label{12.74}\end{equation}
where $\beta_1$ is an isomorphism and $\gamma(V)$ depends analytically on $V$
and maps continuously into
$\dH{\fha-\epsilon}([-2\rho,2\rho]\times\bbS^{n-1}).$ If we consider simply 
the complex multiples of $V,$ i.e.\ just look at $\gamma(zV),$ we have
analyticity in $z.$ The invertibility of this operator reduces to a finite
dimensional problem. Since the map is known to be invertible at $z=0,$
invertibility can only fail at isolated values of $z.$ This proves the result.
\end{proof}

\begin{cor}\label{12.75} For each $\rho>0$ there is a dense
subset of $\dCI(B(\rho))$ near each point of which the backscattering
transform \eqref{12.73} is injective from $\dCI(B(\rho)).$
\end{cor}


\end{document}